\documentclass{article}

\usepackage{amsmath} 
\usepackage{amssymb}  
\usepackage{amsthm}
\usepackage{latexsym}
\usepackage{amsfonts}
\usepackage{xcolor}
\usepackage{graphicx}

\usepackage[letterpaper,top=2cm,bottom=2cm,left=3cm,right=3cm,marginparwidth=1.75cm]{geometry}
\usepackage[colorlinks=true, allcolors=blue]{hyperref}

\newtheorem{thm}{Theorem}

\newtheorem{lem}[thm]{Lemma}
\newtheorem{prop}[thm]{Proposition}
\theoremstyle{definition}
\newtheorem{defn}[thm]{Definition}
\newtheorem{rem}[thm]{Remark}

\newcommand{\Dc}{\mathcal{D}}

\DeclareMathOperator{\gr}{gr}

\newcommand{\braces}[1]{{\rm (}#1{\rm )}}

\renewcommand{\Re}{\operatorname{Re}}

\newcommand{\dom}{\operatorname{dom}}

\newcommand{\ran}{\operatorname{ran}}

\newcommand{\wt}{\widetilde}

\newcommand{\<}{\langle}
\renewcommand{\>}{\rangle}

\newcommand{\R}{\ensuremath{\mathbb R}}    
\newcommand{\C}{\ensuremath{\mathbb C}}    
\newcommand{\K}{\ensuremath{\mathbb K}}    

\newcommand{\calA}{\mathcal A}

\newcommand{\calD}{\mathcal D}

\newcommand{\calH}{\mathcal H}

\newcommand{\calL}{\mathcal L}         
\newcommand{\calM}{\mathcal M}

\newcommand{\calP}{\mathcal P}         
         
\newcommand{\calR}{\mathcal R}

\newcommand{\la}{\lambda}

\newcommand{\bmat}[4]
{
   \begin{bmatrix}
      #1 & #2\\
      #3 & #4
   \end{bmatrix}
}
\newcommand{\bvek}[2]
{
   \begin{bmatrix}
      #1\\
      #2
   \end{bmatrix}
}
\newcommand{\sbvek}[2]{\left[\begin{smallmatrix}#1\\#2\end{smallmatrix}\right]}
\newcommand{\sbmat}[4]{\left[\begin{smallmatrix}#1 & #2\\#3 & #4\end{smallmatrix}\right]}
\newcommand{\bmattt}[9]
{
   \begin{bmatrix}
      #1 & #2 & #3\\
      #4 & #5 & #6\\
      #7 & #8 & #9
   \end{bmatrix}
}
\newcommand{\bvekkk}[3]{\begin{bmatrix}#1\\#2\\#3\end{bmatrix}}
\newcommand{\bvekkkk}[4]{\begin{bmatrix}#1\\#2\\#3\\#4\end{bmatrix}}

\newcommand{\Llra}{\Longleftrightarrow}

\newcommand{\rank}{\operatorname{rank}}
\newcommand{\mul}{\operatorname{mul}}

\title{On the equivalence of geometric and descriptor representations of linear port-Hamiltonian systems\footnote{{\bf Funding:} HG acknowledges funding within the BMBF project EIZ - Project number 03SF0693A. FP was funded by the Carl Zeiss Foundation within the project \textit{DeepTurb--Deep Learning in and from Turbulence} and by the free state of Thuringia and the German Federal Ministry for Education and Research (BMBF) within the project \textit{THInKI--Th\"uringer Hochschulinitiative für KI im Studium}.}}
\author{H. Gernandt$^1$, F. Philipp$^2$, T. Preuster$^2$ and M. Schaller$^2$}
\date{$^1$Fraunhofer Research Institution for Energy Infrastructures and Geothermal Systems IEG Cottbus, Germany\\
$^2$Optimization-based Control Group, Institute for Mathematics, Technische Universit\"at Ilmenau, Germany\\[2ex]%
    \today}
\begin{document}
\maketitle

\begin{abstract}
\noindent We prove a one-to-one correspondence between the geometric formulation of port-Hamiltonian (pH) systems defined by Dirac structures, Lagrange structures, maximal resistive structures, and external ports and a state-space formulation by means of port-Hamiltonian descriptor systems, i.e., differential algebraic equations (DAE) with inputs and outputs.
\end{abstract}

\section{Introduction}
Development and operation of modern technologies requires the deep understanding and control of complex dynamical systems. The class of port-Hamiltonian (pH) systems represents such an elegant mathematical framework for modeling and analysis of multi-physics systems. Due to their inherent energy-based structure, these systems are very well-suited to describe the energy flows, energy conservation and interconnection of physical systems in a wide range of applications. From a modeling perspective, they offer the additional benefit of coupling capability. Port-Hamiltonian systems have found numerous applications in physical domains such as robotics, renewable energy systems, and mechatronics \cite{vdScJelt14,MehU22}.

From a mathematical point of view, there exist different approaches to pH  systems fertilized by different areas of mathematics and mathematical physics. On the one hand one can describe this class of systems by geometrical structures \cite{vdScJelt14}, leading to the concept of Dirac structures. 
Moreover the total energy of the system is given by the Hamiltonian density, which can be generalized by so-called Lagrangian subspaces.

In the language of system and control theory, pH systems can be characterized as descriptor systems with the physical structure of the system being inscribed in the algebraic properties of the coefficient matrices. This perspective allows the application of numerical methods as well as many results from simulation and solution theory, and interprets pH systems as open Hamiltonian systems interacting with their environment by means of inputs and outputs \cite{BeaMXZ18,MehU22}. For example, the pH structure implies certain restrictions on the Kronecker canonical form of the underlying matrix pencil \cite{MehlMehrWojt18} and also provides robustness of the eigenvalues under structured perturbations \cite{MehlMehrWojt20}.

Eventually there is also a functional analytical approach to pH systems theory. This point of view allows the extension of the description of energy-based physical systems on infinite-dimensional state spaces in terms of partial-differential equations and boundary control systems, cf.\ \cite{JacoZwar12} for one-dimensional state domains and \cite{skrepeklinear,PhilReis23} for recent approaches to higher-dimensional state domains.

In this paper, our aim is to reveal a connection between the geometric pH formulation by means of Lagrange structures, Dirac structures, and resistive structures and the system theoretic formulation in finite dimensions by means of input-state-output systems given as a differential-algebraic equation (DAE) of the form
\begin{align}\label{e:DAE_intro}
\begin{bmatrix}\tfrac{\mathrm{d}}{\mathrm{d}t}Ez(t)\\y(t)\end{bmatrix}
=
\begin{bmatrix}
J-R & B-P\\
(B+P)^* & S+N
\end{bmatrix}
\bvek{Qz(t)}{u(t)},\quad t\geq 0
\end{align}
with $\K^m$-valued input $u$ and output $y$, $\K^n$-valued state $z$, and matrices $E,J,R,Q\in\K^{n\times n}$, $B,P\in\K^{n\times m}$, $S,N\in\K^{m\times m}$ having additional structural properties, cf.\ Definition~\ref{d:mehrmann}.

In \cite{MascvdSc18} a first link between the geometric modeling of pH systems outlined in \cite{vdScJelt14} and the state-space representations for DAE-systems from \cite{BeaMXZ18} was established, where the authors only considered Lagrange and Dirac structures without any dissipation or external ports. This results in a state space system of the form
\begin{align}
    \label{eq:KPLS}
K\tfrac{d}{dt}Pz(t)=LSz(t), \quad t\geq 0,
\end{align}
where the matrices $K,L\in\K^{n\times n}$ are given by the kernel representation of the Dirac structure and $P,S\in\K^{n\times n}$ are given by a range representation of the Lagrange structure. 

The case with dissipation was considered in \cite{GernHR21} for pH descriptor systems. To this end, the Dirac structure was replaced by a dissipative subspace and for structural results on the underlying matrix pencils nonnegativity of the Lagrange structure was assumed.

The geometric setting in \cite{GernHR21} was further generalized in  \cite{mehrmann2023differential} where, contrary to the dissipative subspace in \cite{GernHR21}, in addition to the Dirac and Lagrange structures a resistive structure was used to model the dissipation. The relation to state space systems of the form \eqref{eq:KPLS} was studied and the index as well as the Kronecker canonical form of \eqref{eq:KPLS} was investigated. However no external port variables were considered. 

Recently, in \cite{van2022linear} a geometric description of dissipative pH descriptor systems including port variables was given. It was shown that the previously used geometric definition of pH systems, either via a separate resistive structure or via a dissipative structure (called {\em monotone} in \cite{van2022linear}) are in fact equivalent. This was used to obtain a state space formulation \eqref{e:DAE_intro} from the geometric description using a monotone structure. 

The main contribution of the present note is to also provide a converse result, i.e.\ for pH descriptor systems \eqref{e:DAE_intro} satisfying $\ker E\cap\ker Q=\{0\}$ we derive an equivalent geometric formulation. This extends previous results from \cite{MehrMora19}, where no additional Lagrange structure was considered, leading to a one-to-one correspondence of the two formulations in the behavioral sense. This means that for each solution of the geometric pH descriptor system there is a corresponding solution of the state-space formulation and vice versa. 

Furthermore, in comparison to \cite{van2022linear} we show that geometric pH systems have a state space formulation \eqref{e:DAE_intro}, where $Q$ equals the identity, which is often assumed in pH literature \cite{MehU22}. Incorporating the converse direction, it follows that each pH descriptor system is equivalent to another one in a possibly larger state-space with $Q = I$.

The paper is organized as follows: In Section~\ref{sec:prelim} we recall notations and well known facts from multi-valued linear algebra. After presenting both the geometric and the descriptor formulation of pH descriptor systems in Section~\ref{sec:defn}, the one-to-one correspondence between these two formulations is shown in Section~\ref{sec:equivalence}. We conclude the paper and discuss open problems in Section~\ref{sec:conclusion}.

\section{Preliminaries from multi-valued linear algebra}
\label{sec:prelim}
\noindent\textbf{Notation:} $\K$ denotes either $\C$ or $\R$---consistently throughout this article. The graph $\{(x,Ax) : x\in\K^n\}$ of a linear map $A : \K^n\to\K^m$ is denoted by $\gr A$. Its inverse (as a linear relation) is given and denoted by $\gr^{-1}\!A = \{(Ax,x) : x\in\K^n\}$. For $A\in\K^{n\times n}$ we write $A^*:=\overline{A}^\top$ where $\overline{A}$ is the entry-wise complex conjugate of $A$, i.e.\ if $A\in\R^{n\times n}$ we have $A^*=A^\top$. The Euclidean inner product in $\K^n$ will be denoted by $\langle x,y\rangle:=y^*x$ for all $x,y\in\K^n$ with the resulting Euclidean norm $\|x\|^2:=\langle x,x\rangle$.

Recall the notions of kernel, domain, multivalued part, and range of a linear subspace of a product space (also called \textit{linear relation}).

\begin{defn}\label{d: dirac_subsp}
The {\em kernel, domain, multivalued part,} and {\em range} of a linear subspace $\calA\subset\K^n\times\K^m$ are defined by
\begin{align*}
\ker\calA &:= \{f\in\K^n : (f,0)\in\calA\,\},\\
\dom\calA &:= \{f\in\K^n : \exists\, e\in\K^m\text{ s.t. }(f,e)\in\calA\,\},\\
\mul\calA &:= \{e\in\K^m : (0,e)\in\calA\,\},\\
\ran\calA &:= \{e\in\K^m : \exists\, f\in\K^n\text{ s.t. }(f,e)\in\calA\,\},
\end{align*}
respectively. The \emph{inverse} $\calA^{-1}$,  the {\em adjoint} $\calA^*$, and \emph{scalar multiples} $\alpha \calA$ of $\calA$ are defined by 
\begin{align*}
\calA^{-1} &:= \{(e,f) : (f,e)\in\calA\},\\
\qquad
\calA^* &:= \left\{(e',f') : \<f',e\> = \<e',f\>\;\forall (e,f)\in\calA\right\},\\
\alpha\calA &:= \{(e,\alpha f) : (e,f)\in\calA\},\quad \alpha\in\K.
\end{align*}
\end{defn}

In particular, non-invertible matrices $A$ can be inverted in the sense of linear relations by considering $\gr^{-1}A$ which might then be multi-valued if $A$ is not injective or not-everywhere defined if $A$ is not surjective.

In the following, we collect some notions for subspaces which have additional structural properties.
\begin{defn}
\label{def:subspace_properties}
Let $\calD$, $\calL$, $\calM$, and $\calR$ be subspaces of $\K^{2n}$.
\begin{itemize}
    \item[\rm (i)] $\calL$ is called a \emph{Lagrange structure} if $\calL = \calL^*$.
    \item[\rm (ii)]  $\calD$ is called a \emph{Dirac structure} if $\calD = -\calD^*$.
  \item[\rm (iii)]  $\calR$ is called a \emph{(maximal) resistive structure} if $\calR\subset\calR^*$,
   \[
\<e,f\> \leq  0\quad \text{for all $\sbvek ef\in\calR$\quad  (and $\dim \calR=n$).} 
 \]
 \item[\rm (iv)]  $\calM$ is called a \emph{(maximal) monotone structure} if
 \[
\Re\<e,f\> \geq  0\quad \text{for all $\sbvek ef\in\calM$\quad  (and $\dim \calM=n$);} 
 \]

 \end{itemize}
\end{defn}

\begin{rem}
{\bf (a)} In the language of linear relations slightly different nomenclature is used. Dirac, Lagrange and (maximal) monotone structures are called skew-adjoint, self-adjoint and (maximal) accretive, respectively, cf.\ \cite{BehrHdS20}. A resistive structure would be called a non-positive symmetric relation.

\smallskip\noindent
{\bf (b)} $\calL$ is a Lagrange structure if and only if $\calL\subset\calL^*$ and $\dim\calL = n$. In particular, a maximal resistive structure is also a Lagrange structure. Similarly, $\calD$ is a Dirac structure if and only if $\calD\subset -\calD^*$ and $\dim\calD = n$.
\end{rem}

\begin{rem}
In the case $\K = \C$, resistive structures always have maximal resistive extensions. This follows directly from \cite[Theorem 5.3.1]{BehrHdS20}. Similarly, monotone structures always have maximal monotone extensions. Indeed, if $\calM$ is a monotone structure in $\C^{2n}$, then its Cayley transform is a linear contraction $V : \dom V\to\C^n$ with $\dom V\subset\C^n$, cf.\ \cite[Proposition 1.6.6]{BehrHdS20}. Let $\tilde V : \C^n\to\C^n$ be a contractive extension of $V$. Then the inverse Cayley transform of $\tilde V$ is an extension as desired.
\end{rem}

There are two common representations which will be referred to in this article as \emph{kernel} and \emph{image representation}, see \cite[Theorem 3.3]{BerWT16} and also \cite{BehrHdS20}.
\begin{prop}
Let $\mathcal{M}$ be a subspace of $\K^{2n}$ of dimension $d$. Then there exists matrices $K,L\in\K^{(2n-d)\times n}$ and $F,G\in\K^{n\times d}$ such that the following holds
\begin{align}
\label{eq:ker_ran}    
\calM=\ker[K,L]=\ran\begin{bmatrix}
    F\\G
\end{bmatrix}.
\end{align}
\end{prop}
If $\calD=\ker[K,L]$ is a Dirac structure for some $K,L\in \K^{n\times n}$, then \cite{MascvdSc18} used the notion \emph{Dirac algebraic constraint} if $K$ is not invertible. This is equivalent to the existence of $(z,0)\in\Dc$ with $z\neq 0$, or in the language of linear relations $\ker \Dc\neq \{0\}$. A special case of such constraints are kinematic constraints (see Example 2.7 in \cite{MascvdSc18}). Analogously, for a Lagrange structure $\calL=\ran\sbvek PS$ for some  $P,S\in\K^{n\times n}$ there are said to be \emph{Lagrange algebraic constraints} if $P$ is not invertible. These can be used to model algebraic state constraints. 

The matrices in the kernel and range representations \eqref{eq:ker_ran} can be used to characterize the structural properties from Definition~\ref{def:subspace_properties}. In the next proposition, we restrict ourselves to the range representation.

\begin{prop}
\label{prop:ran_properties}
Let $\calM$ be a subspace of $\K^{2n}$ which is given by $\calM=\ran\sbvek PS$ for some $P,S\in\K^{n\times m}$. Then the following equivalences hold:
\begin{itemize}
    \item[\rm (i)] $\calM=\ran\sbvek PS$ is a Lagrange structure if and only if $S^*P=P^*S$ and $\rank\sbvek PS=n$;
    \item[\rm (ii)] $\calM=\ran\sbvek PS$ is a Dirac structure if and only if $S^*P=-P^*S$ and $\rank\sbvek PS=n$;
    \item[\rm (iii)] $\calM=\ran\sbvek PS$ is (maximal) monotone if and only if $S^*P+P^*S\geq 0$ (and $\rank\sbvek PS=n$);
    \item[\rm (iv)] $\calM=\ran\sbvek PS$ is (maximal) resistive if and only if $S^*P=P^*S\leq 0$ (and $\rank\sbvek PS=n$).
\end{itemize}
\end{prop}

\begin{rem}
Similar characterizations as in Proposition \ref{prop:ran_properties} can also be derived for subspaces $\calM$ given in kernel representations  $\calM=\ker[K,L]$ for some $K,L\in\K^{(2n-d)\times n}$. Then the adjoint relation $\calM^*$ is given in range representation 
\[
\mathcal{M}^*=\ran\sbvek{L^*}{-K^*}.
\]
Furthermore, $\calM$ is Lagrange (resp.\ Dirac, maximal monotone, maximal resistive) if and only if $\calM^*$ has this property. Indeed, Lagrange and maximal resistive structures satisfy $\calM=\calM^*$, Dirac structures satisfy $\calM=-\calM^*$, see e.g.\ \cite{GernHR21} and it was shown in  \cite[Proposition 1.6.7]{BehrHdS20}
that $\calM$ is maximal monotone if and only if $\calM^*$ is maximal monotone.

Therefore, can apply Proposition~\ref{prop:ran_properties} to $\calM^*$ which implies that $\calM=\ker[K,L]$ is Lagrange (resp.\ Dirac, maximal monotone, maximal resistive) if and only if ($KL^*=LK^*$, $KL^*=-LK^*$, $KL^*+LK^*\leq 0$, $KL^*=LK^*\ge 0$).
\end{rem}

The following result is the key to rewrite maximal subspaces as graphs of matrices in a larger subspace.

\begin{prop}\label{p:extended}
Let $\calM\subset\K^N\times\K^N$ be a Dirac (resp.\ Lagrange, maximal resistive, maximal monotone) structure and let $l := \dim\ker\calM$. Then there exist matrices $G\in\K^{N\times l}$ with $\ker G = \{0\}$ and $M\in\K^{N\times N}$ satisfying 
$M=-M^*$ (resp.\ $M=M^*$, $M=M^*\leq 0$, $M+M^*\geq 0$) and $M\ker G^*\subset\ker G^*$ such that
\begin{align}
\label{eq:M_representation}
\calM = \left\{\bvek{Me-G\la}{e} : G^*e=0,\,\la\in\K^l\right\}.
\end{align}
\end{prop}
\begin{proof}
The claim has been proven for Dirac structures in \cite[Proposition 3.8]{van2022linear} (see also \cite[Theorem 3.1]{dalsmo1998representations}),  for Lagrange structures in \cite[Proposition 5.3]{MascvdSc18} and for maximal resistive and monotone subspaces in \cite{GernHR21}.
\end{proof}
We conclude this section with a remark.

\begin{rem}\hfill
\begin{itemize}
\item[\rm (i)] If $\sbvek xy\in\calM$, where $\calM$ is as in \eqref{eq:M_representation}, then both $e$ and $\la$ are uniquely determined: $e = y$ and $\la = G^\dagger(My - x)$, where $G^\dagger$ is any left-inverse of $G$.
\item[\rm (ii)] Since the inverse relation $\calM$ of Dirac, Lagrange, maximal resistive and maximal monotone structures inherits the particular property, we can apply Proposition~\ref{p:extended} to $\calM^{-1}$ and obtain the existence of $\hat M$ and an injective $\hat G$ such that
\[
\calM = \left\{\bvek{e}{\hat Me-\hat G\la} : \hat G^*e=0,\,\la\in\K^{\hat l}\right\}.
\]
\end{itemize}
\end{rem}

\section{Two formulations of pH systems}\label{sec:defn}
In this part we introduce the two formulations of pH systems we will consider in the remainder of this work. In the upcoming Subsection~\ref{subsec:geom} we introduce the geometric representation, whereas in Subsection~\ref{subsec:dae} we recall the formulation by means of a differential algebraic descriptor system.
\subsection{Geometric representation of pH systems}\label{subsec:geom}
The following geometric description of pH systems was recently introduced in \cite{van2022linear} and extends the geometric formulation from \cite{MascvdSc18} by incorporating resistive variables, inputs and outputs.

\begin{defn}[{\bf \cite{van2022linear}}]\label{def:implicit_PH}
A {\em geometric representation} of a pH system \braces{in short: a {\em geometric pH system}} with state space $\K^n$ and external dimension $m$ is given by a triple $(\calD,\calL,\calR)$ consisting of
\begin{itemize}
\item a Dirac structure $\Dc\subseteq\K^{n+r+m}\times\K^{n+r+m}$,
\item a Lagrange structure $\calL\subset\K^n\times\K^n$, and
\item a maximal resistive structure $\calR\subset\K^r\times\K^r$.
\end{itemize}
By a {\it solution} of the pH system $(\calD,\calL,\calR)$ we understand an input-state-output trajectory $(u,x,y) \in C([0,\infty);\K^m)\times C^1([0,\infty);\K^n)\times C([0,\infty);\K^m)$ for which there exist continuous functions $f_R$, $e_R$, and $e_L$ such that for all $t\ge 0$ we have
\begin{equation}\label{e:solution_implicit}
\big(-\dot x(t),f_R(t),y(t),e_L(t),e_R(t),u(t)\big)\in\mathcal{D},\quad (x(t),e_L(t))\in\mathcal L,\quad  (f_R(t),e_R(t))\in\mathcal{R}.
\end{equation}
The functions $f_R$ and $e_R$ are called the {\em resistive flow and effort variables}, respectively, and $e_L$ is the {\em Lagrangian effort}.
\end{defn}

We briefly comment on this definition in view of previous works and generalizations.
\begin{rem}
In a more general setting, the  structures in Definition~\ref{def:implicit_PH} might also depend on time $t$ and state $x$, cf.\ \cite{vdScJelt14}. However, here we only consider stationary structures. Furthermore, the variables $u$ and $y$, usually denoting inputs and outputs in systems and control theory frameworks, were called $f_P$ and $e_P$ in \cite{van2022linear}, respectively. Last, we note that in, e.g., \cite[Definition 2.1]{van2022linear} or in \cite[Defintion 14]{mehrmann2023differential}, negated maximal resistive structures are called non-negative Lagrange structures. To avoid confusion with the Lagrange structure $\calL$, we will utilize the notion maximal resistive structure for $\calR$.
\end{rem}

Solutions of geometric pH systems obey a {\em power-balance}, as the following elementary result shows.

\begin{lem}
Let $(u,x,y) \in C([0,\infty);\K^m)\times C^1([0,\infty);\K^n)\times C([0,\infty);\K^m)$ be a solution of the geometric pH system $(\calD,\calL,\calR)$. Then, for all $t\geq 0$, the following power balance holds:
\begin{align*}
\Re\<\dot x(t),e_L(t)\> = \Re\<f_R(t),e_R(t)\> + \Re\<y(t),u(t)\> \leq \Re\<y(t),u(t)\>.
\end{align*}
\end{lem}
\begin{proof}
Since $\calD$ is a Dirac structure as defined in Definition~\ref{def:subspace_properties}(ii), and due to \eqref{e:solution_implicit}, we compute for all $t\geq 0$
\begin{align}
\label{eq:implicit_PBE}
\begin{split}
0=\Re \left\langle
\begin{bmatrix}
    -\dot x(t)\\ f_R(t)\\ y(t)
\end{bmatrix},\begin{bmatrix}
   e_L(t)\\ e_R(t)\\ u(t)
\end{bmatrix}\right\rangle
&= -\Re\<\dot x(t),e_L(t)\> + \Re\<f_R(t),e_R(t)\> + \Re\<y(t),u(t)\>.
\end{split}
\end{align}
We have $\Re\<f_R(t),e_R(t)\>\le 0$ as $(f_R(t),e_R(t))\in\calR$ and $\calR$ is resistive, cf.\ Definition~\ref{def:subspace_properties}(iii).
\end{proof}

\subsection{Port-Hamiltonian descriptor systems}\label{subsec:dae}
A second formulation of linear pH systems is given in a somewhat more explicit form, involving a differential-algebraic equation (DAE), see e.g.\ \cite[Definition 4.9]{MehU22}, see also \cite{BeaMXZ18,MehrMora19}.

\begin{defn}[{\bf \cite{MehU22}}]\label{d:mehrmann}
A pH descriptor system is a DAE with inputs and outputs of the form
\begin{align}\label{e:DAE}
\begin{bmatrix}\tfrac{\mathrm{d}}{\mathrm{d}t}Ez(t)\\y(t)\end{bmatrix}
=
\begin{bmatrix}
J-R & B-P\\
(B+P)^* & S+N
\end{bmatrix}
\bvek{Qz(t)}{u(t)}
\end{align}
with $\K^m$-valued input $u$ and output $y$, $\K^n$-valued state $z$, matrices $E,J,R,Q\in\K^{n\times n}$, $B,P\in\K^{n\times m}$, $S,N\in\K^{m\times m}$ satisfying
$$
E^*Q = Q^*E,\quad J=-J^*,\quad N=-N^*,\quad R=R^*,\quad S=S^*
$$
such that
\begin{align}\label{e:Wpositive}
W:=\begin{bmatrix}
    Q^*&0\\0&I
\end{bmatrix}\begin{bmatrix} R&P\\ P^* &S\end{bmatrix}\begin{bmatrix}
    Q&0\\0&I
\end{bmatrix}\geq 0.
\end{align}
The {\em Hamiltonian} of the system is defined as $H(z)=z^*Q^*Ez$. A solution of \eqref{e:DAE} is an input-state-output trajectory $(u,z,y)\in C([0,\infty);\K^{n+2m})$ with $Ez\in C^1([0,\infty);\K^n)$ such that \eqref{e:DAE} is satisfied for all $t\geq 0$.
\end{defn}
Note that one could also generalize the above definition to inputs $u\in L_{\mathrm loc}^1((0,\infty);\K^m)$ when considering $Ez \in W^{1,1}_{\mathrm{loc}}([0,\infty);\K^n)$.

The following result yields a regularity result of the Hamiltonian along solutions and power balance for the DAE system \eqref{e:DAE}. Its proof follows by straightforward modifications of \cite[Lemma 2.2]{FaulMasc22}, where a similar result was shown for solutions in $W^{1,1}_{\rm loc}([0,\infty);\K^m)$ and we state it here for completeness.

\begin{lem}\label{lem:hamiltonian}
If $(u,z,y)$ is a solution of \eqref{e:DAE}, then $H\circ z\in C^{1}([0,\infty);\K^n)$, and the following power balance holds:
\begin{align}
\label{eq:energybalance}
\frac{\mathrm{d}}{\mathrm{d}t} H(z(t)) = \Re\big[u(t)^* y(t)\big] - \left\|W^{\frac12}\begin{bmatrix}
    z(t)\\u(t)
\end{bmatrix}\right\|^2.
\end{align}
\end{lem}
\begin{proof}
Let $\mathcal{P}$ denote the orthogonal projection onto $\ran E^*$, i.e., $I-\mathcal{P}$ maps onto $(\ran E^*)^\perp = \ker E$. Hence, have $E = E\mathcal{P} + E(I-\mathcal{P}) = E\mathcal{P}$. Let $E^\dagger$ denote the Moore-Penrose inverse of $E$. Then $\mathcal{P} = E^\dagger E$ and therefore
$$
Ez\in C^1([0,\infty),\K^n)
\quad\Llra\quad
\mathcal{P}z\in C^1([0,\infty),\K^n).
$$
Since $E^*Q=Q^*E$, we have $H(z) = \frac 12 z^*\calP E^*Qz = \frac 12 (\calP z)^*Q^*E(\calP z)$ and thus $H\circ z\in C^1([0,\infty);K^n)$. Consequently, and as $\Re(z^*Q^*JQz) = \Re(u^*Nu) = 0$, we obtain
\begin{align*}
\frac{\mathrm{d}}{\mathrm{d}t}(H\circ z)
&= \Re\left[\left(\frac{\mathrm{d}}{\mathrm{d}t}\calP z\right)^* Q^*E\calP z\right] = \Re\left[\left(\frac{\mathrm{d}}{\mathrm{d}t}\mathcal{P}z\right)^*E^*Qz\right] = \Re\left[\left(\frac{\mathrm{d}}{\mathrm{d}t}Ez\right)^*Qz\right]\\
&= \Re\big[(J-R)Qz + (B-P)u\big]^*Qz = -z^*Q^*RQz + \Re(u^*(B-P)^*Qz)\\
&= \Re\big[u^*\big((B+P)^*Qz + (S+N)u\big) - 2u^*P^*Qz - u^*Su\big] - z^*Q^*RQz\\
&= \Re\big[u^*y\big] - \begin{bmatrix}z^*&u^*\end{bmatrix}W \begin{bmatrix}z\\u\end{bmatrix},
\end{align*}
which is the claimed power balance.
\end{proof}

\section{Equivalence of the two formulations}
\label{sec:equivalence}
In this section, we associate a pH descriptor system in the sense of Definition~\ref{d:mehrmann} with a geometric pH system as defined in Definition~\ref{def:implicit_PH} and vice-versa. This shows that the two formulations introduced in Section~\ref{sec:defn} are equivalent.

\subsection{From geometric pH to descriptor pH}
The next theorem shows that geometric pH systems $(\calD,\calL,\calR)$ can be associated with particular pH descriptor systems such that solutions of the geometric pH system are uniquely determined parts of solutions of the descriptor system and vice versa.

\begin{thm}
\label{thm:diractovolker}
Let a geometric pH system $(\calD,\calL,\calR)$ be given as in Definition \ref{def:implicit_PH} and set $p = \dim\ker\calD + \dim\ker\calR+\dim\ker\calL$. Then there exists a pH descriptor system of the form
\begin{align}\label{e:mehrmann_in_thm}
\begin{bmatrix}\tfrac{d}{dt} E z(t)\\y(t)\end{bmatrix}
=
\begin{bmatrix}
	 J- R & B\\
	B^* & N
\end{bmatrix}
\bvek{z(t)}{u(t)}
\end{align}
as in \eqref{e:DAE} with $Q=I$, $P=0$, $S=0$ with the state $z \in\K^{n+r+p}$ such that the following hold:
\begin{enumerate}
\item[{\rm (i)}] If $(u,x,y)$ is a solution of $(\calD,\calL,\calR)$ then there exists $z$ such that $(u,z,y)$ solves \eqref{e:mehrmann_in_thm}.
\item[{\rm (ii)}] If $(u, z,y)$ is a solution of \eqref{e:mehrmann_in_thm}, then for every $(x_0,e_L(0))\in\calL$ there exists $x$ such that $(u,x,y)$ solves $(\calD,\calL,\calR)$ with $x(0)=x_0$. 
\end{enumerate}
Furthermore, if $-\calL$ is resistive, then $ E=E^*\geq 0$ holds.
\end{thm}
\begin{proof}
Let $d = \dim\ker\calD$, $k = \dim\ker\calR$, $l = \dim\ker\calL$, and $N = n+r+m$. By Proposition~\ref{p:extended}, there exist an injective $G\in\K^{N\times d}$ and a skew-adjoint $\wt J\in\K^{N\times N}$ such that
\begin{equation}
\calD = \left\lbrace \bvek{\wt Je - G\la}{e} : G^*e=0,\,\la\in\K^d,\,e\in\K^N\right\rbrace.
\end{equation}
Let $(u,x,y)$ be a solution of $(\calD,\calL,\calR)$ with $f_R$, $e_R$, and $e_L$ as in \eqref{e:solution_implicit}, i.e.,
\begin{equation}\label{e:in_dirac}
 \big(-\dot x(t),f_R(t),y(t),e_L(t),e_R(t),u(t)\big)\in\mathcal{D},\quad (x(t),e_L(t))\in\mathcal L,\quad  (f_R(t),e_R(t))\in\mathcal{R}.
\end{equation}
Hence, we find that there exists $\la : [0,\infty)\to\K^d$ such that
\begin{equation}\label{e: matrix_repr}
\begin{bmatrix}-\dot x(t)\\f_R(t)\\y(t)\\0\end{bmatrix}
= \bmat{\wt J}{-G}{G^*}{0} \begin{bmatrix}e_L(t)\\e_R(t)\\u(t) \\ \lambda(t) \end{bmatrix}
= \begin{bmatrix}
	J_{11} & J_{12} & J_{13} & -G_1 \\
	-J_{12}^* & J_{22} & J_{23} & -G_2\\
	-J_{13}^* & -J_{23}^* & J_{33} & -G_3 \\
	G_1^* & G_2^* & G_3^* & 0
\end{bmatrix}
\begin{bmatrix}e_L(t)\\e_R(t)\\u(t) \\ \lambda(t) \end{bmatrix}.
\end{equation}
Making use of Proposition \ref{p:extended} again, we find that the maximal resistive structure $\calR$ has a representation
\begin{align}
\label{eq:R_rep}
\calR = \left\{\bvek {- \wt Rx_R + G_R\lambda_R}{x_R} : G_R^*x_R=0,\,\lambda_R\in\K^k,\,x_R\in\K^r\right\},
\end{align}
where $G_R\in\K^{r\times k}$ is injective and $\wt R\in\K^{r\times r}$ is a positive semi-definite Hermitian matrix. Hence, \eqref{eq:R_rep} implies
\[
f_R(t) = -\wt R e_R(t) + G_R\lambda_R(t),\qquad G_R^*e_R(t)=0.
\]
Therefore, \eqref{e: matrix_repr} can be equivalently rewritten as
\[
\begin{bmatrix}\dot x(t)\\0\\y(t)\\0\\0\end{bmatrix}
= \begin{bmatrix}
	-J_{11} & -J_{12} & -J_{13} & G_1 &0\\
	J_{12}^* & -J_{22}- \wt R & -J_{23} & G_2 &G_R\\
	-J_{13}^* & -J_{23}^* & J_{33} & -G_3&0 \\
	-G_1^* & -G_2^* & -G_3^* & 0&0\\ 0&-G_R^*&0&0&0
\end{bmatrix} \begin{bmatrix}e_L(t)\\e_R(t)\\u(t) \\ \lambda(t) \\ \lambda_R(t) \end{bmatrix}
 \]
and after an additional permutation of the rows and columns we obtain
\begin{align}\label{eq:systemwithR}
\begin{bmatrix}\dot x(t)\\0\\0\\0\\y(t)\end{bmatrix}
= \begin{bmatrix}
	-J_{11} & -J_{12}  & G_1 &0& -J_{13}\\
	J_{12}^* & -J_{22}-\wt R & G_2 &G_R & -J_{23}\\
	-G_1^* & -G_2^*  & 0&0& -G_3^*\\ 0&-G_R^*&0&0&0\\-J_{13}^* & -J_{23}^*  & -G_3&0 & J_{33}
\end{bmatrix} \begin{bmatrix}e_L(t)\\e_R(t) \\ \lambda(t) \\ \lambda_R(t)\\ u(t) \end{bmatrix}.
 \end{align}
Leveraging Proposition~\ref{p:extended} one more time, we may express the Lagrange structure $\calL$ as
\begin{align}\label{eq:L_rep}
\calL = \left\{\bvek {Lx_L - G_L\lambda_L}{x_L} : G_L^*x_L=0,\,\lambda_L\in\K^l,\,x_L\in\K^n\right\},
\end{align}
where $G_L\in\K^{n\times l}$ is injective and $L\in\K^{n\times n}$ is a Hermitian matrix. Hence, for given $e_L(t)\in\K^n$ there exists a unique $\lambda_L(t)\in\K^l$ satisfying
\begin{align}
\label{eq:x_e_replaced}    
x(t)=Le_L(t)-G_L\lambda_L(t),\qquad G_L^*e_L=0
\end{align}
By Proposition~\ref{p:extended}, $Le_L(t)$ is orthogonal to $G_L\lambda_L(t)$ for all $t\geq 0$ and therefore $x\in C^1([0,\infty),\mathbb{K}^n)$ holds if and only if $Le_L,G_L\lambda_L\in C^1([0,\infty),\mathbb{K}^n)$ holds. Moreover, as $G_L$ is injective, we have $G_L^\dagger G_L = I_l$ and so $G_L\lambda_L\in C^1([0,\infty),\K^n)$ is equivalent to $\lambda_L\in C^1([0,\infty),\K^l)$. Using \eqref{eq:x_e_replaced}, the system \eqref{eq:systemwithR} is equivalent to
\begin{align*}
\begin{bmatrix}\tfrac{d}{dt}Le_L(t)-G_L\dot\lambda_L(t)\\0\\0\\0\\y(t)\end{bmatrix}
= \begin{bmatrix}
	-J_{11} & -J_{12}  & G_1 &0& -J_{13}\\
	J_{12}^* & -J_{22}- R & G_2 &G_R & -J_{23}\\
	-G_1^* & -G_2^*  & 0&0& -G_3^*\\ 0&-G_R^*&0&0&0\\-J_{13}^* & -J_{23}^*  & -G_3&0 & J_{33}
\end{bmatrix} \begin{bmatrix}e_L(t)\\e_R(t) \\ \lambda(t) \\ \lambda_R(t)\\ u(t) \end{bmatrix}
 \end{align*}
which can be rewritten as
\begin{align}
\label{eq:final_system}
\begin{bmatrix}\tfrac{d}{dt}Le_L(t)\\0\\0\\0\\0\\y(t)\end{bmatrix}
= \begin{bmatrix}
	-J_{11} & -J_{12}  & G_1 &0& G_L& -J_{13}\\
	J_{12}^* & -J_{22}- \wt R & G_2 &G_R &0 & -J_{23}\\
	-G_1^* & -G_2^*  & 0&0&0& -G_3^*\\ 0&-G_R^*&0&0&0&0\\-G_L^*&0&0&0&0&0 \\-J_{13}^* & -J_{23}^*  & -G_3&0 &0& J_{33}
\end{bmatrix} \begin{bmatrix}e_L(t)\\e_R(t) \\ \lambda(t) \\ \lambda_R(t)\\ \dot\lambda_L(t)\\ u(t) \end{bmatrix}.
\end{align}
We now define 
\begin{align}
\label{eq:def_dae_matrices}
\begin{split}
J&:=\begin{bmatrix}
	-J_{11} & -J_{12}  & G_1 &0& G_L\\
	J_{12}^* & -J_{22} & G_2 &G_R &0 \\
	-G_1^* & -G_2^*  & 0&0&0\\ 0&-G_R^*&0&0&0\\-G_L^*&0&0&0&0 
\end{bmatrix},
\quad R:=\begin{bmatrix}
	0 & 0  &0&0&0\\
	0 & \wt R & 0 &0 &0 \\
	0 & 0  & 0&0&0\\ 0&0&0&0&0\\0&0&0&0&0 
\end{bmatrix},
\quad E=\begin{bmatrix} 
    L&0&0&0&0\\0&0&0&0&0\\0&0&0&0&0\\0&0&0&0&0\\0&0&0&0&0
        \end{bmatrix},\\
B&:=\begin{bmatrix}
    -J_{13}\\ -J_{23}\\ -G_3^*\\0\\0
\end{bmatrix},\quad N=J_{33} ,\quad S=0,\quad Q=I,\quad P=0.
\end{split}
\end{align}
If $(u,x,y)$ solves $(\calD,\calL,\calR)$, then $Le_L\in C^1([0,\infty),\K^n)$ holds and $(u,z,y)$ with $z=(e_L,e_R,\lambda,\lambda_R,\dot\lambda_L)$ solves \eqref{eq:final_system}. This proves (i). To show (ii), let  $(u,z,y)$ with $z=(e_L,e_R,\lambda,\lambda_R,\mu_L)$ solve  \eqref{eq:final_system}. Then for $(x_0,e_L(0))\in\calL$ there exists unique $\lambda_L^0\in\K^l$ such that $x_0=Le_L(0)+G_L\lambda_L^0$ holds. We set $\lambda_L(t)=\lambda_L^0+\int_0^t \mu_L(s)\,ds$ and define $x(t)=Le_L(t)-G_L\lambda_L(t)$. Since $Le_L(0)$ is given, we have $x(0)=Le_L(0)+G_L\lambda_L^0 = x_0$. Furthermore,   $f_R(t)=-\wt Re_R(t)+G_R\lambda_R(t)$ fulfills $(f_R(t),e_R(t))\in\calR$ and consequently \eqref{e:solution_implicit} is satisfied. Therefore, $(u,x,y)$ is a solution of $(\calD,\calL,\calR)$ which proves (ii).

If $-\calL$ is resistive, then Proposition~\ref{p:extended} implies that $L$ as in \eqref{eq:L_rep} fulfills $L=L^*\geq 0$ and therefore $E$ given by \eqref{eq:def_dae_matrices} satisfies $E=E^*\geq 0$.
\end{proof}

\begin{rem}
Assume that for the Lagrange structure $\calL$ in Definition~\ref{def:implicit_PH} the subspace $-\calL$ is resistive. Theorem \ref{thm:diractovolker} yields $E\geq 0$. Therefore, the pencil $(E,J-R)$ is regular\footnote{I.e., $\lambda E-(J-R)$ is invertible for some $\lambda\in\mathbb{C}$}, if and only if $\ker E\cap\ker (J-R)=\{0\}$ holds, see e.g.\ \cite[Lemma 6.1.4]{BergerPhD}. Furthermore, it was shown in \cite{BergerPhD} that the following holds
\[
(sE-(J-R))^{-1}+(sE^*-(J-R)^*)^{-1}\geq 0\quad \text{for $\Re(s)> 0$.}
\]
Therefore the resulting descriptor system is positive real, i.e.\ that the transfer function $G(s)=B^*(sE-(J-R))^{-1}B+N$ fulfills
\[
G(s)+G(s)^*=B^*((sE-(J-R))^{-1}+(sE^*-(J-R)^*)^{-1})B\geq 0\quad \text{for $\Re(s)> 0$.}
\]
More details on the relation of positive real and pH descriptor systems, as well as their relation to passive descriptor systems can be found in \cite{CheGH22}.
\end{rem}

\begin{rem}
In the proof of Theorem \ref{thm:diractovolker}, 
we applied Proposition \ref{p:extended} to each of the subspaces
$\calD$, $\calL$, and $\calR$.
Since Proposition \ref{p:extended} holds for monotone subspaces as well, one might generalize and assume that $\calL$ is monotone. In this case, the corresponding pH descriptor system can be derived as in Theorem \ref{thm:diractovolker} and $E+E^*\geq 0$ holds. Then, the resulting matrix pair  $(E,J-R)$ is said to have positive Hermitian part. In~\cite{MehMW22}, the spectral properties as well as the regularity and the Kronecker canonical form of these pencils were further analysed.
\end{rem}

\subsection{From descriptor pH to geometric pH}
Next, we show how to associate a geometric pH system with a given pH descriptor system such that there is a one-to-one correspondence between the solutions of the two systems. The following result is a slight extension of \cite[Theorem 3]{MehrMora19} where no additional Lagrange structure was considered.

\begin{thm}\label{thm:desc_to_imp}
Let a pH descriptor system as in Definition \ref{d:mehrmann} be given with
\[
\ker E\cap\ker Q = \{0\}.
\]
Let $W = \sbmat RP{P^*}S$ as in \eqref{e:Wpositive}, $\Gamma := \sbmat JB{-B^*}{-N}$, set $r = n+m$, and define
$$
\calL := \ran\bvek EQ\subset\K^n\times\K^n,\qquad \calR := \gr(-W) = \ran\bvek{I_r}{-W}\subset\K^r\times\K^r.
$$
Further, with the matrices
$$
U := \bmattt{I_n}0000{I_r}0{I_m}0\in\K^{(n+r+m)\times (n+r+m)}
\qquad\text{and}\qquad \wt D := \bmat{-\Gamma}{-I_r}{I_r}0\in\K^{(n+r+m)\times (n+r+m)}
$$
define the subspace
$$
\calD := \gr^{-1}\!\big(U\wt DU^*\big).
$$
Then $(\calD,\calL,\calR)$ is a geometric pH system. Moreover, the following hold:
\begin{enumerate}
\item[{\rm (i)}] If $(u,z,y)$ solves the DAE \eqref{e:DAE}, then $(u,Ez,y)$ solves the geometric pH system $(\calD,\calL,\calR)$.
\item[{\rm (ii)}] If $(u,x,y)$ solves the geometric pH system $(\calD,\calL,\calR)$ with $f_R$, $e_R$, and $e_L$ as in Definition \ref{def:implicit_PH}, then $(u,z,y)$ solves the DAE \eqref{e:DAE}, where $z$ is the unique function satisfying $x=Ez$ and $e_L=Qz$.
\end{enumerate}
\end{thm}
\begin{proof}
Assume that $(u,z,y)$ solves \eqref{e:DAE}. Then, setting
\[
e_L := Qz,\quad f_R := \bvek{Qz}u,\quad e_R := -Wf_R,
\]
we have
\begin{align*}
	\bvekkk{-\dot x}y{f_R}
	&= \bvekkkk{-\frac d{dt}Ez}y{Qz}u = \bvekkkk{(-J+R)Qz + (-B+P)u}{(B+P)^*Qz + (S+N)u}{Qz}{u} = \bvek{(-\Gamma + W)f_R}{f_R} = \bvek{-\Gamma f_R - e_R}{f_R} = \wt D\bvekkk{Qz}u{e_R}.
\end{align*}
Now applying $U$ from the left to this equation shows that $(-\dot x,f_R,y,e_L,e_R,u)\in\calD$.

Conversely, if $(u,x,y)$ solves the geometric pH system $(\calD,\calL,\calR)$, then there exist functions $f_R,e_L,e_R$ such that $(-\dot x,f_R,y,e_L,e_R,u)\in\calD$, $(x,e_L)\in\calL$, and $(f_R,e_R)\in\calR$. By the definition of $\calL$ and $\calR$, there exists a unique function $z$ such that $x= Ez$ and $e_L = Qz$. Moreover, $e_R = -Wf_R$, and we obtain
$$
\bvekkk{-\frac d{dt}Ez}y{f_R} = \bvekkk{-\dot x}y{f_R} = \wt D\bvekkk{e_L}u{e_R} = \bmat{-\Gamma}{-I_r}{I_r}0\bvekkk{Qz}u{-Wf_R} = \bvek{-\Gamma\bvek{Qz}u + Wf_R}{\bvek{Qz}u}.
$$
This implies that $f_R = \sbvek{Qz}u$ and thus
$$
\bvek{-\frac d{dt}Ez}y = (W-\Gamma)\bvek{Qz}u,
$$
which means that $(u,z,y)$ solves \eqref{e:DAE}.
\end{proof}

\begin{rem}
For pH descriptor system as in Definition~\ref{d:mehrmann} satisfying  $\ker E\cap\ker Q=\{0\}$, we obtain a geometric pH system from Theorem~\ref{thm:desc_to_imp}. Applying Theorem \ref{thm:diractovolker} to this geometric system, leads back to a pH descriptor system as in Definition~\eqref{d:mehrmann} which fulfills $Q=Id$. Hence, we obtain a pH descriptor system with invertible $Q$, that is equivalent to the original descriptor system, but with a larger state space dimension, see also \cite{MehlMehrWojt20,MehU22} for alternative methods on achieving invertibility of $Q$. 
\end{rem}

\section{Conclusion, extensions and open problems}
\label{sec:conclusion}
We have shown that the geometric formulation of port-Hamiltonian systems is equivalent to the state-space representation by means of differential algebraic equations. To this end, we utilized tools from multilinear algebra and provided constructive proofs to transfer either of the formulation to the other. The main assumption to derive a geometric representation from a pH-DAE was that the matrices in the Hamiltonian $\calH(x)=\tfrac12x^*Q^*Ex$, $x\in\mathbb{K}^n$, satisfy $\ker E\cap\ker Q=\{0\}$. 

Concerning future research, a first extension could be the investigation of the case of a non-trivial kernel intersection. In this case, one cannot directly define a Lagrange subspace by means of these two matrices as it was done in the proof of Theorem~\ref{thm:desc_to_imp}. One possible remedy could be to isolate the common kernel by means of a common singular value decomposition of $E$ and $Q$, similarly as in  
\cite[Lemma 3.6]{ilchmann2023port}. Furthermore, the geometric pH representation obtained in Theorem~\ref{thm:desc_to_imp} could be  not optimal in the sense that the dimensions of constructed Dirac and the maximal resistive structure, which is equal to $2(n+m)$, might be further reduced.

In view of an extension to the infinite-dimensional case, a first step would be to define an infinite-dimensional differential algebraic formulation of pH systems. Whereas infinite-dimensional DAEs are a very delicate issue \cite{FavY98,reis2006systems,SviridyukFedorov2003,SeifertTW2022}, a definition for closed systems was given in~\cite{jacob2022solvability} by incorporating the pH structure, i.e., the dissipativity of the main operator. From a geometric point of view, infinite-dimensional pH systems give rise to (Stokes)-Dirac structures~\cite{van2002hamiltonian} including also boundary port variables. An analytical viewpoint on Dirac structures for skew-symmetric differential operators was provided in \cite{le2005dirac,villegas2007port}.

In the present note, we focused on continuous time pH systems. Recently, in \cite{CheGHM23} a definition for  discrete-time pH descriptor systems was given. Therein the discrete time pH system was obtained from a Cayley transformation of continuous time systems. Furthermore,  the Cayley transform was applied to the underlying Dirac structures which results in contractive subspaces as discrete-time counter parts of Dirac structures. It remains an open problem to compare the geometric pH formulation and the descriptor pH formulation for discrete-time systems.

\bibliographystyle{alpha}
\bibliography{sample}

\end{document}